\documentclass[11pt]{amsart}
\usepackage{amssymb}
\usepackage{url}

\title[On Poonen's conjecture]{On Poonen's conjecture concerning rational preperiodic points of quadratic maps}
\author{Benjamin Hutz and Patrick Ingram$^\dag$}
\address{Department of Mathematics and Computer Science, Amherst College, Amherst, MA 01002}
\email{bhutz@amherst.edu}
\address{Department of Pure Mathematics, University of Waterloo, Waterloo, Ontario, Canada N2L 3G1}
\email{pingram@math.uwaterloo.ca}
\date{\today}
\thanks{$^\dag$ The work of the second author was supported in part by a grant from NSERC of Canada.}

\newcommand{\QQ}{\mathbb{Q}}
\newcommand{\ZZ}{\mathbb{Z}}

\newcommand{\PP}{\mathbb{P}}

\renewcommand{\AA}{\mathbb{A}}

\newcommand{\Ocal}{\mathcal{O}}

\newcommand{\Gal}{\operatorname{Gal}}

\newcommand{\lcm}{\operatorname{lcm}}
\newcommand{\PosPer}{\operatorname{PosPer}}

\newcommand{\MOD}[1]{~(\textup{mod}~#1)}
\renewcommand{\phi}{\varphi}
\newcommand{\pf}{\mathfrak{p}}

\providecommand{\abs}[1]{\left\lvert#1\right\rvert}
\newcommand{\col}{\,{:}\,}

\newtheorem{proposition}{Proposition}

\newtheorem{lemma}{Lemma}
\newtheorem{conj}{Conjecture}

\theoremstyle{remark}
\newtheorem*{remark}{Remark}
\newtheorem*{example}{Example}

\begin{document}
\maketitle

\begin{abstract}
    The purpose of this note is to give some evidence in support of conjectures of Poonen and Morton and Silverman on the periods of rational numbers under the iteration of quadratic polynomials. In particular, for the family of maps $f_c(x) = x^2 + c$ for $c \in \QQ$, Poonen conjectured that the exact period of a $\QQ$-rational periodic points is at most $3$.  Using good reduction information, we verify this conjecture over $\QQ$ for $c$ values up to height $10^8$. For $K/\QQ$ a quadratic number field, we provide evidence that the upper bound on the exact period of a $\QQ$-rational periodic point is $6$.  We also show that the largest exact period of a rational periodic point increases at least linearly in the degree of the number field.
\end{abstract}

\section{Introduction}
The purpose of this note is to give some evidence in support of conjectures of Poonen and Morton and Silverman on the periods of rational numbers under the iteration of quadratic polynomials.
Suppose that $\phi_c(z)=z^2+c$, where $c\in\QQ$.  We will say that $\alpha\in\PP^1(\QQ)$ is a periodic point with exact period $n$ for $\phi_c$ if $\phi_c^n(\alpha)=\alpha$, while $\phi_c^m(\alpha)\neq\alpha$ for $0<m<n$.  For example, the point at infinity is a point with exact period $1$ for $\phi_c$, as are
$$\alpha=\frac{1}{2}\pm\frac{\sqrt{1-4c}}{2}$$
in the case where $1-4c$ is a rational square.  It is equally easy to construct infinite families of quadratic polynomials $\phi_c$ with points with (exact) period 2 or 3, but there are no such polynomials with points of period 4, 5, or (assuming certain conjectures) 6 \cite{fps, morton, stoll}.  It is reasonable to ask for which $N$ there exists a pair $\alpha, c\in\QQ$ such that $\alpha$ is a point with exact period $N$ for $\phi_c$ (or how many rational periodic points a quadratic map defined over $\QQ$ may have, in total; the questions are explicitly related \cite{poonen}).

Let $E/\QQ$ be an elliptic curve, and let $[\ell]:E\rightarrow E$ be the multiplication-by-$\ell$ map, for some prime $\ell$.  Then the preperiodic points for $[\ell]$ (that is, the points $P$ such that $[\ell]^m(P)$ is periodic, for some $m$) are precisely the torsion points on $E$.  A theorem of Mazur \cite{mazur} demonstrates that, for any given elliptic curve $E/\QQ$, a point of finite order in $E(\QQ)$ has order at most 12.  For a fixed prime $\ell$, this bounds the period of a rational periodic point of $[\ell]:E\rightarrow E$ in terms of $\ell$, but independent of $E$.  Poonen conjectured a similar bounded for the (exact) period of periodic points for quadratic polynomials.

\begin{conj}[Poonen \cite{poonen}]\label{poonen's conjecture}
There is no $c\in\QQ$ such that $\phi(z)=z^2+c$ has a $\QQ$-rational periodic point with exact period greater than 3.
\end{conj}

More generally, a result of Kamienny \cite{kamienny} shows that if $K/\QQ$ is a quadratic extension, and $E/K$ is an elliptic curve, any point of finite order in $E(K)$ has order at most 18.  A still more general theorem of Merel \cite{merel} shows that for any finite extension $K/\QQ$, the order of torsion points in $E(K)$ is bounded in terms of $[K:\QQ]$ alone.  In light of these more general results, one might make the following conjecture, a special case of a general conjecture of Morton and Silverman \cite{morton-silverman}.

\begin{conj}
For any $d\geq 1$ there is some $M=M(d)$, which we take to be minimal, such that for any number field $K/\QQ$ of degree at most $d$, and any $c\in K$, the quadratic polynomial $\phi_c$ has no point defined over $K$ with exact period greater than $M$.
\end{conj}

The aim of this note is to present computational evidence for Conjectures~1 and 2.  We use the notation $$h\left(\frac{p}{q}\right) = \log(\max(\abs{p},\abs{q})) \quad \text{with} \quad \gcd(p,q) = 1,$$ to denote the standard logarithmic height function on $\QQ$.  For $K=\QQ(\omega)$ a quadratic extension and $c \in K$ we write $c=a \omega + b$ and define $h(c) = \max(h(a),h(b))$.  We denote the non-logarithmic height function as $H(c)$.

\begin{proposition}\label{prop:poonen}
There is no $c\in\QQ$ such that $\phi_c(z)=z^2+c$ has a $\QQ$-rational periodic point with exact period greater than 3 and $h(c)\leq 8\log 10$.
\end{proposition}
In other words, if $\phi(z)=z^2+c$ violates Poonen's Conjecture, then either the numerator or the denominator of $c$ exceeds $10^8$ in absolute value.  Something similar is true for quadratic extensions of $\QQ$.

\begin{proposition}\label{prop:mortonsilv}
For any quadratic field $K/\QQ$ with discriminant $D$ satisfying $-4000\leq D\leq 4000$, and any $c\in K$ with $h(c)\leq 3\log 10$, the quadratic polynomial $\phi_c(z)=z^2+c$ has no $K$-rational point with exact period greater than $6$.
\end{proposition}

Finally, it is interesting to consider the intermediate question, that is, which periods are possible for $z\mapsto z^2+c$ in an extension $K/\QQ$, when $c$ is $\QQ$-rational.  Here the computational evidence is somewhat stronger than when $c$ is allowed to vary in $K$.

\begin{proposition}\label{prop:thirdway}
For any quadratic field $K/\QQ$ with discriminant $D$ satisfying $-4000\leq D\leq 4000$, and any $c\in\QQ$ satisfying $h(c)\leq 6\log 10$, the quadratic polynomial $\phi_c(z)=z^2+c$ has no $K$-rational point with exact period greater than $6$.
\end{proposition}

We conjecture, then, that the minimal bound $M(d)$ from Conjecture~2 is $M(1)=3$ and $M(2)=6$.  These are certainly the smallest possible values, given that $z^2-\frac{29}{16}$ has the 3-cycle
\[\left[-\frac{1}{4}\rightarrow -\frac{7}{4} \rightarrow \frac{5}{4}\right]\rightarrow -\frac{1}{4}\rightarrow\cdots,\]
 and $z^2-\frac{71}{48}$ has the 6-cycle
\begin{multline*}
\Big[-\frac{1}{4}+\frac{1}{6}\sqrt{33}\rightarrow -\frac{1}{2}-\frac{1}{12}\sqrt{33}\rightarrow -1+\frac{1}{12}\sqrt{33}\rightarrow -\frac{1}{4}-\frac{1}{6}\sqrt{33}\\ \rightarrow -\frac{1}{2}+\frac{1}{12}\sqrt{33}\rightarrow -1-\frac{1}{12}\sqrt{33}\Big]\rightarrow -\frac{1}{4}+\frac{1}{6}\sqrt{33}\rightarrow\cdots.
\end{multline*}
Note that, since the field obtained by adjoining to $\QQ$ a point of period $N$ for the map $z\mapsto z^2+c$ has degree at most $2^N$, it is clear that our quantity $M(d)$ must satisfy $M(d)\geq \log_2(d)$.
\begin{proposition}
    $M(d) \geq \frac{d}{3}$.
\end{proposition}
\begin{proof}
    By examining $p^n$-th roots of unity $\zeta_{p^n}$ for an odd prime $p$ and $\phi(z)=z^2$, we see that
\[
d_n = [\QQ(\zeta_{p^n}) \col \QQ]=p^{n-1}(p-1)
\]
and hence
\[
    M(d_n) \geq \alpha p^{n-1} = \frac{\alpha}{p-1}d_n,
\]
    where $\alpha$ is the multiplicative order of $2$ modulo $p$.  Taking $p=3$ we get $\alpha = p-1$ and hence $M(d_n) \geq d_n$.  Since $M(d)$ is nondecreasing, we know that for all $n \in \mathbb{N}$, $M(d) \geq d_{n}$ for $d_{n} \leq d < d_{n+1}$.  In particular, the line through the points $(2\cdot 3^n,2\cdot 3^{n-1})$ for $n \in \mathbb{N}$ is always below the graph of $M(d)$.  This is the line $y = \frac{x}{3}$.
\end{proof}

\begin{remark}
    In \cite{poonen}, Poonen assumes that $M(1)=3$ and shows that there can be at most $9$ $\QQ$-rational preperiodic points, which is achieved by $c = -\frac{29}{16}$.  Manes \cite{manes} considered a certain family of degree-two rational maps on $\PP^1$, and showed that if the largest possible exact period of a $\QQ$-rational periodic point for such a map is $4$, then there are at most $12$ $\QQ$-rational preperiodic points.  In the quadratic field case for the quadratic polynomials $\phi_c$, there can be at least $15$ $K$-rational preperiodic points achieved for $c=-\frac{29}{16}$ and the field $K=\QQ(\sqrt{17})$.
\end{remark}

%%%%%%%%%%%%%%%%%%%%%%%%%%%%%%%

\section{The verification of Propositions~\ref{prop:poonen}, \ref{prop:mortonsilv}, and~\ref{prop:thirdway}}

\emph{A priori}, it is not obvious that it is a finite computation to check either conjecture for even a single value of $c$.  A simple argument, due to Northcott, shows that all periodic points for $z^2+c$ lie in a set of bounded height (the bound depends upon the height of $c$).  Thus, one may in principle find all of the periodic points for $z^2+c$ rational over the ground field (indeed, of bounded degree).  In practice, however, this computation is too time-consuming to execute for more than a handful  of values of $c$.

One approach to the analogous problem for elliptic curves, to quickly bound the rational torsion, is to reduce modulo one or more primes of good reduction.  Since the reduction-modulo-$p$ map is injective on torsion (for elliptic curves with good reduction), comparing the orders of the group of points on a given elliptic curve over a few different finite fields offers a fairly efficient way of finding the order of torsion on the curve globally.  A similar approach is effective in the current context.
\begin{lemma}[{\cite[p.~62]{ads}}]\label{periods}
Let $\phi:\PP^1_K\rightarrow\PP^1_K$ be a rational function of degree $\geq 2$, defined over a local field with valuation $v$.  Assume that $\phi$ has good reduction, let $P\in\PP^1(K)$ be a periodic point of $\phi$, and define the following quantities:
\begin{itemize}
\item[$n$] the exact period of $P$ for $\phi$
\item[$m$] the exact period of the reduced point $\tilde{P}$ for the reduced function $\tilde{\phi}$
\item[$r$] the order of $\lambda=(\tilde{\phi}^m)'(\tilde{P})$ in $k^*$, where $k$ is the residue field (set $r=\infty$ if $\lambda$ is not a root of unity)
\item[$p$] the characteristic of the residue field.
Then $n$ has one of the following forms:
\[n=m, \qquad n=mr, \qquad\text{ or }\qquad n=mrp^e,\]
for some positive integer $e$.
If $K$ has characteristic 0, and $v$ is normalized in the usual way, then
\[p^{e-1}\leq\frac{2v(p)}{p-1}.\]
\end{itemize}
\end{lemma}

 Thus, if $K$ is a number field there are, for each prime $\pf$ of good reduction for $\phi$, only a finite list of possible numbers of the form $m$, $mr$, or $mrp^e$ of the form above.  If we call this finite set the set of \emph{possible periods} for $\phi$ determined by $\pf$, we may simply intersect many of these sets and hope that the information confirms Conjecture~\ref{poonen's conjecture}.  In practice, this method is efficient for this work with most $c$ values requiring the possible periods for fewer than $10$ primes and nearly all $c$ values requiring fewer than $30$.

Let $K$ be a number field, with ring of integers $\Ocal_K$, and let $S$ be a finite set of  primes of $\Ocal_K$.  By an $S$-\emph{type} of $K$, we mean an element of
\[\prod_{\pf\in S}\PP^1(\Ocal_K/\pf).\]
The $S$-type of an element $x\in\PP^1(K)$ is determined by the usual reduction-modulo-$\pf$ map $\PP^1_K\rightarrow\PP^1_{\Ocal_K/\pf}$ in each coordinate.
For each set $S$, and each $S$-type $T$ for which at least one coordinate is not the point at infinity, Lemma~\ref{periods} provides a method for describing a finite list of possible periods of points for $z^2+c$, for any $c\in K$ with that particular $S$-type. It is worth noting that the $\pf$-coordinate of the $S$-type of $c$ is the point at infinity if and only if $\phi_c$ has bad reduction at $\pf$.

More precisely, for each periodic point $\alpha\in\Ocal_K/\pf$ under iteration by $z\mapsto z^2+c$, let $m(\alpha)$ denote the corresponding period, and $r(\alpha)$ the multiplicative order of the multiplier.  We let $\PosPer(c)$ denote the set of natural numbers containing each $m(\alpha)$ and each number of the form $m(\alpha)r(\alpha)p^e$ for $1\leq p^{e-1}\leq \frac{2v(p)}{p-1}$.  For $c$ the point at infinity in $\PP^1(\Ocal_K/\pf)$, we simply set $\PosPer(c)=\mathbb{N}$.
More generally, for each finite set of primes $S$, and each $S$-type $T$, we can construct the set $\PosPer(T)$ of possible periods for $c\in\Ocal_K$ with $S$-type $T$ by setting
\[\PosPer(T)=\bigcap_{\pf\in S}\PosPer(T_\pf),\]
where $T_\pf\in \PP^1(\Ocal_K/\pf)$ is the $\pf$-coordinate of $T$.
  It follows from Lemma \ref{periods} that if $P\in K$ is a point of exact period $m$ for $z\mapsto z^2+c$, then $m\in\PosPer(T)$, where $T$ is the $S$-type of $c$, for any fixed, finite set $S$ of primes.  By construction, the set $\PosPer(T)$ will be finite as long as at least one $T_\pf$ is not the point at infinity; in other words, as long as any $c\in K$ with $S$-type $T$ has good reduction at some prime in $S$.

    We have omitted the case where $P$ reduces to the point at infinity in $\PP^1(\Ocal_K/\pf)$.  This is not problematic, however; it turns out (see Lemma~\ref{walde} below) that any affine point reducing to the infinite point modulo a prime of \emph{good} reduction $\pf$, necessarily has an unbounded forward orbit in the $\pf$-adic topology.  Such a point cannot, of course, be periodic.

The following is also a useful,  simple, observation made in  \cite{walde}.
\begin{lemma}\label{walde}
Suppose that $z^2+c$ has a periodic point $\alpha\in\AA^1(K)$.  Then for each nonarchimedean place $v$ of $K$ with $v(c)<0$, we have $v(c)=2v(\alpha)$.  For each nonarchimedean place with $v(c)\geq 0$, we have $v(\alpha)\geq 0$.
\end{lemma}

\begin{proof}
 If $0>v(c)>2v(\alpha)$, then
$v(\alpha^2+c)=2v(\alpha)$, by the ultrametric inequality, and so in particular $2v(\alpha^2+c)<v(\alpha^2+c)<v(c)$.  By induction, $v(\phi^n(\alpha))= 2^nv(\alpha)$ which, since $v(\alpha)\neq 0$, contradicts the periodicity of $\alpha$.  If, on the other hand, $0>v(c)$ and $2v(\alpha)>v(c)$, we have $v(\alpha^2+c)=v(c)$.  But in this case, $2v(\phi(\alpha))=2v(c)<v(c)$, and so the previous argument shows that $\phi(\alpha)$ (and hence $\alpha$) is not periodic.

For the second claim, simply note that if $v(c)\geq 0$ but $v(\alpha)<0$, we immediately conclude $v(\alpha^2+c)=2v(\alpha)<0$.  By induction we obtain $v(\phi^n_c(\alpha))=2^nv(\alpha)$, from which is it clear that $\alpha$ cannot be preperiodic under $\phi_c$.
\end{proof}

In other words, if $\mathcal{D}(c)$ is the ideal of $\Ocal_K$ defined by
\[\mathcal{D}(c)=\prod_{\pf}\pf^{\max\{0, -v_\pf(c)\}},\]
the map $z\mapsto z^2+c$ can have no $K$-rational periodic points (other than the point at infinity) unless $\mathcal{D}(c)$ is a square.  The following lemma simplifies matters further:

\begin{lemma}
Let $K/\QQ$ be a quadratic field with discriminant $D$, and let
\[\omega=\frac{a+\sqrt{D}}{2}, \qquad a=\begin{cases}1 & \text{ if }D\equiv 1\MOD{4}\\ 0 &\text{ if }D\equiv 0\MOD{4},\end{cases}\]
so that $1, \omega$ is a basis for $\Ocal_K/\ZZ$.  Then for any integers $a, b, c, d$, with $\gcd(a, b)=\gcd(c, d)=1$, we have $\mathcal{D}(\frac{a}{b}+\omega\frac{c}{d})$ a square only if $\lcm(b, d)$ is a square (as an ideal in $\Ocal_K$).
\end{lemma}

The statement above does \emph{not} hold locally.  That is, there might be valuations $v$ on $\Ocal_K$ such that $v(\frac{a}{b}+\omega\frac{c}{d})<0$ is even, but such that $v(\lcm(b, d))$ is odd.  For example, if $D=-4$ and $v$ is the valuation corresponding to the prime ideal $(2+\omega)\Ocal_K$, then $\frac{2}{125}+\omega\frac{1}{125}$ has the property that $v(\lcm(b, d))=v(5^3)=3$ is odd,  but $v(\mathcal{D}(\frac{2}{125}+\omega\frac{1}{125}))=2$.  Whenever this is the case,  there must be some other valuation $v'$ such that $v'(\frac{a}{b}+\omega\frac{c}{d})<0$ is odd (in the example, we may take $v'$ to be the valuation corresponding to the other divisor of 5).

\begin{proof}
Let $v$ be a valuation on $\Ocal_K$.  First, suppose that $v$ ramifies in $\Ocal_K/\ZZ$.  Then $v(\lcm(b, d))$ is automatically even.

Now, suppose that $v$ is inert in this extension.  Then we certainly have $v(\omega)=0$.
Suppose that $v(\lcm(b, d))$ is odd.  In this case, if $v(\frac{a}{b}+\omega\frac{c}{d})<0$ is even, then either $v(b)>0$ or $v(d)>0$.  If exactly one of these is positive, then $v(ad+bc\omega)=0$, and so
\[v\left(\frac{a}{b}+\omega\frac{c}{d}\right)=-v(bd)=-v(\lcm(b, d)).\]
The claim follows.  Now suppose that $v(b)>0$ and $v(d)>0$.  In this case,
\[v\left(\frac{a}{b}+\omega\frac{c}{d}\right)=\min(v(a/b), v(c\omega/ d))=-v(\lcm(b, d))\]
\emph{unless} $v(b)=v(d)$.  So we are left with the case $v(d)=v(b)$.  If $p$ is the unique rational prime with $v(p)=1$, then $d=p^ed_0$ and $b=p^eb_0$, for some $d_0$ and $b_0$ prime to $p$, and $e=v(d)=v(b)$.  We claim that $v(ad_0+cb_0\omega)=0$.  If not, then $\omega\equiv -\frac{ad_0}{cb_0}\MOD{p}$.  But if this were the case, then the image of $\omega$ in $\Ocal_K/p\Ocal_K$ would be an element of degree one over the prime field $\ZZ/p\ZZ\subseteq\Ocal_K/p\Ocal_K$, which is clearly a contradiction, since $\Ocal_K=\ZZ[\omega]$.  It must be that $v(ad_0+cb_0\omega)=0$, whence
\[v\left(\frac{a}{b}+\omega\frac{c}{d}\right)=e-v(bd)=-e=-v(\lcm(b, d)).\]

Finally, suppose that $v$ is split,  let $\pf\subseteq\Ocal_K$ be the corresponding prime, let $p$ be the rational prime below $\pf$, let $\sigma$ generate $\Gal(K/\QQ)$, and let $v'$ be the valuation associated to $\pf^\sigma$.
Suppose that $v(b)>v(d)$, and that $v(\lcm(b, d))$ is odd.  Then we certainly have $v(cb\omega)>v(d)=v(ad)$, and so
\[v(ad+bc\omega)=v(d).\]
It follows that
\[v\left(\frac{a}{b}+\omega\frac{c}{d}\right)=v(d)-v(bd)=-v(\lcm(b, d)).\]
So we may suppose that $v(b)\leq v(d)$.  In particular, if $e=v(b)$, we may write $b=p^eb_0$, $d=p^ed_0$, with $b_0, d_0\in\ZZ$, and $b_0$ prime to $p$.  It cannot be the case that
\[\min(v(ad_0+cb_0\omega), v'(ad_0+cb_0\omega))>0.\]
If this were the case, then we would have
\[ad_0+cb_0\omega\equiv 0\MOD{\pf}\text{ and }ad_0+cb_0\omega^\sigma\equiv 0\MOD{\pf}.\]
From this, it would follow that $cb_0(\omega-\omega^\sigma)\equiv 0\MOD{\pf}$, which is impossible since $p\nmid cb_0$, and since $p$ is not ramified.  Thus we have either
\[v(ad_0+cb_0\omega)=0\text{ or }v'(ad_0+cb_0\omega)=0,\]
whereupon either
\[v\left(\frac{a}{b}+\omega\frac{c}{b}\right)=v(ad_0+cb_0\omega)+v(\gcd(b, d))-v(bd)=-v(\lcm(b, d)),\]
or the same statement for $v'$.  Either way, if $v(\lcm(b, d))$ is odd, then some valuation on $K$ is negative and odd at $\frac{a}{b}+\frac{c}{d}\omega$.
\end{proof}

We may now describe the algorithm.
\begin{enumerate}
\item Choose a small initial set $S$ of primes and, for each $S$-type $T$, compute $\PosPer(T)$.
\item Construct a list $\mathfrak{S}$ of all $S$-types $T$ such that $\PosPer(T)\not\subseteq\{1, ..., M\}$, for the conjectured value $M=M([K:\QQ])$.
\item Loop over $c\in K$ with $H(c)\leq B$, writing $c$ as $\frac{\alpha}{\beta}+\frac{\gamma}{\delta}\omega$ (taking $\omega=0$ if $K=\QQ$):
\begin{enumerate}
\item if $\lcm(\beta, \delta)$ is not a square ideal in $\Ocal_K$ then we know that $z^2+c$ has no $K_v$-rational periodic points for some completion $K_v$ of $K$: the conjecture holds trivially.
\item compute the $S$-type $T(c)$: if $T(c) \not\in \mathfrak{S}$, then we know that the conjecture holds for $c$.
\item if $T(c)\in \mathfrak{S}$ then we may refine our data for this value of $c$ by computing the set of possible periods modulo additional primes of good reduction.
\item if after a suitably large number of additional primes $\PosPer(c) \not\subseteq \{1,\ldots,M\}$, then the case is left to be treated manually.
\end{enumerate}
\end{enumerate}

In practice, the two initial simplifying steps take care of the vast majority of cases.  For example, the number of $c\in \QQ$ with $H(c)\leq B$ is roughly $2B^2$, but once the $c$ with non-square denominators are eliminated, only $2B^{3/2}$ of these need to be checked.

 Pre-computing the data for several primes improves things even more.  In turns out that for most congruence classes modulo a given prime, there are local conditions ensuring the truth of Poonen's Conjecture.  For example, over $\QQ$, there are 3 distinct $\{2\}$-types, namely $(0)$, $(1)$, and $(\infty)$.  It is fairly easy to see that Conjecture~1 holds if $c\equiv 0\MOD{2}$ or $c\equiv 1\MOD{2}$, since $z\mapsto z^2+1$ has a single, superattracting 2-cycle over $\ZZ/2\ZZ$, and $z\mapsto z^2$ has two superattracting fixed points\footnote{These cycles should probably be called ``wildly superattracting'', since the derivative of the map vanishes at every point.}.
 So, already our precomputation has eliminated two-thirds of possible $c\in\QQ$.
The data below shows the proportion of types that need to be considered if one pre-computes data for the first $N$ primes over $\QQ$ (recall that $\mathfrak{S}$ is the set of types for which $\PosPer\not\subseteq\{1, 2, 3\}$, i.e., those congruence classes which require extra computation to verify the conjecture).
\begin{center}
\begin{tabular}{|c|r|r|c|}\hline
$N$ & $\# \mathfrak{S}$ & number of types & proportion\\\hline\hline
1 & 1 & 3 & 0.33333\\\hline
2 & 2 & 12 & 0.16667\\\hline
3 & 5  & 72 & 0.06944\\\hline
4 & 13 & 576 & 0.02257\\\hline
5 & 40 & 6912 & 0.00579\\\hline
6 & 98 & 96768 & 0.00101\\\hline
7 & 199 & 1741824 & 0.00011\\\hline
8 & 862 & 34836480 & $2.4744\times 10^{-5}$\\\hline
9 & 1699 & 836075520 & $2.0321 \times 10^{-6}$\\\hline
10 & 4893 & 25082265600 & $1.9508\times 10^{-7}$\\\hline
%11 & 28397 & 802632499200 & $3.5380\times 10^{-8}$\\\hline
\end{tabular}
\end{center}

There is a similar phenomenon for each quadratic extension $K/\QQ$ but the numerics are not included here as the proportions vary for each field.

\section{Other computational issues}
    While computing $\mathfrak{S}$ for many primes reduces the proportion of types left to check, the more primes used does not necessarily lead to faster execution.  The main impediment is that for each $c$ value whose denominator is a square you still must compute the $S$-type (and hash value) for that $c$ value.  The more primes in $S$, the more operations you must perform for every $c$ value.  Since most $c$ values have an $S$-type not in $\mathfrak{S}$, the number of extra operations becomes a significant factor.  For example, we found that $\#S=5$ to be the most efficient over $\QQ$ for the authors' implementation of the algorithm.

    However, even computing $\PosPer(c)$ for infinitely many primes may not give you the precise list of exact periods for periodic points for that $c$ value.
    \begin{example}
        For $K/\QQ$ with discriminant $-3$, and $c = 1/4(\omega + 1)$, where $\omega = \frac{1 + \sqrt{-3}}{2}$.  The point $\omega/2$ is a fixed point with multiplier $\omega$.  Hence, in the notation of Lemma \ref{periods}, for all primes of good reduction we have $m=1$ and $r=6$.  Therefore, for any set of primes $S$ we have $\{1,6\} \subseteq \PosPer(c)$, yet there are no periodic points with exact period $6$ for $\phi_c$.
    \end{example}

    It is interesting to note that no $c$ values were found to have a periodic point with exact period $5$. The only point found with exact period $6$ was previously known \cite[Table 6]{fps}.
    \begin{remark}
        Since the curves parameterizing $(x,c)$ where $x$ is a periodic point with exact $5$ or $6$ for $\phi_c$ both have genus greater than $1$ \cite{fps}, we know by Faltings' Theorem that there are at most finitely many $c$ values which admit a periodic point of exact period $5$ or $6$ over $K/\QQ$ with $[K:\QQ] \leq 2$.  It would be interesting to determine the complete set of periodic points with exact period $5$ or $6$ over $K$, in particular, whether or not there are any $c$ values which admit a periodic point of exact period $5$ over $K$.
    \end{remark}

    The code was written in $C$ and run on the Amherst College computing cluster.  Sample code is available at \url{http://www.amherst.edu/~bhutz/Research.html}.

\end{document}